\documentclass{amsart}

\usepackage{epsf,latexsym,amsmath,amsfonts,amssymb,subfigure,mathrsfs,graphicx,stmaryrd,eufrak,esint,amscd}
\usepackage{color}

\theoremstyle{plain}
\newtheorem{theorem}{Theorem}[section]
\newtheorem{lemma}[theorem]{Lemma}

\newtheorem{coro}[theorem]{Corollary}
\newtheorem{hypo}[theorem]{Hypothesis}

\theoremstyle{definition}

\theoremstyle{remark}
\newtheorem{remark}[theorem]{Remark}

\numberwithin{equation}{section}
\newcommand{\abs}[1]{\lvert#1\rvert}
\newcommand{\labs}[1]{\left\lvert\,#1\,\right\rvert}

\newcommand{\Lr}[1]{\left(#1\right)}
\newcommand{\lr}[1]{\bigl(#1\bigr)}
\newcommand{\set}[2]{\{\,#1\,\mid\,#2\,\}}

\newcommand{\nm}[2]{\|\,#1\,\|_{#2}}
\newcommand{\jump}[1]{[\![#1]\!]}
\newcommand{\wnm}[1]{|\!|\!|#1|\!|\!|_{\iota,h}}
\newcommand{\red}{\color{red}}

\newcommand{\mc}[1]{\mathcal{#1}}
\newcommand{\mb}[1]{\mathbb{#1}}
\newcommand{\ov}[1]{\overline{#1}}

\def\al{\alpha}

\def\del{\delta}

\def\na{\nabla}
\def\eps{\epsilon}
\def\pa{\partial}
\def\Om{\Omega}

\def\lam{\lambda}
\def\x{\times}
\def\vpi{\varPi}

\def\md{\mathrm{d}}

\def\C{\mathbb C}
\def\D{\mathbb D}
\def\R{\mathbb R}

\def\dx{\,\mathrm{d}x}
\def\dy{\,\mathrm{d}y}
\def\dsx{\,\mathrm{d}\sigma(x)}

\def\osc{\text{Osc}}

\DeclareMathOperator{\divop}{\na\cdot}

\def\negint{{\int\negthickspace\negthickspace\negthickspace
\negthinspace -}}

\begin{document}
\title[Nonconforming elements for Strain Gradient Model]{H$^2-$Korn's Inequality and the Nonconforming Elements for The Strain Gradient Elastic Model}

\author{Hongliang Li}
\address{Department of Mathematics, Sichuan Normal University, Chengdu 610066, China\\
 email: lhl@sicnu.edu.cn}

\author{Pingbing Ming}
\address{LSEC, Institute of Computational Mathematics and Scientific/Engineering Computing, AMSS\\
 Chinese Academy of Sciences, No. 55, East Road Zhong-Guan-Cun, Beijing 100190, China\\
 and School of Mathematical Sciences, University of Chinese Academy of Sciences, Beijing 100049, China\\
 email: mpb@lsec.cc.ac.cn}
\author{Huiyu Wang}
\address{Beijing 101 Middle School\\
email: why@lsec.cc.ac.cn}
%
\begin{abstract}
We establish a new H$^2-$Korn's inequality and its discrete analog, which greatly simplify the construction of nonconforming elements for a linear strain gradient elastic model. The Specht triangle~\cite{Specht:1988} and the NZT tetrahedron~\cite{WangShiXu:2007} are analyzed as two typical representatives for robust nonconforming elements in the sense that the rate of convergence is independent of the small material parameter. We construct new regularized interpolation estimate
and the enriching operator for both elements, and prove the error estimates under minimal smoothness assumption on the solution. Numerical results are consistent with the theoretical prediction.
\end{abstract}

\keywords{Strain gradient elasticity, H$^2-$Korn's inequality, robust finite elements}
\date{\today}
\maketitle
\section{Introduction}
Let $u$ be the solution of the following boundary value problem:
\begin{equation}\label{eq:sgbvp}
\left\{
\begin{aligned}
(\iota^2\triangle-I)\Lr{\mu\triangle u+(\lambda+\mu)\na\divop u}&=f\quad&&\text{in\;}\Om,\\
u=\pa_n u&=0\quad&&\text{on\;}\pa\Om,
\end{aligned}\right.
\end{equation}
where $\lam$ and $\mu$ are the Lam\'e constants, and $\iota$ is the microscopic
parameter satisfying $0<\iota\le 1$. In particular, we are interested in the regime when $\iota$ is close
to zero. This boundary value problem arises from a  linear strain gradient elastic model proposed by Aifantis et al~\cite{Altan:1992, Ru:1993}, and $u$ is the displacement. This model may be regarded as a simplification of the more general strain gradient elasticity models~\cite{Mindlin:1964} because it contains only one extra material parameter $\iota$ besides the Lam\'e constants 
$\lambda$ and $\mu$. This simplified strain gradient model successfully eliminated the strain singularity of the brittle crack tip field~\cite{Exadaktylos:1996}, and we refer to~\cite{Eringen:2002} and~\cite{FleckHutchinson:1997} for other strain gradient models.

The boundary value problem~\eqref{eq:sgbvp} is essentially a singularly perturbed elliptic system of fourth order due to the appearance of the strain gradient $\na\eps(u)$. C$^1$-conforming finite elements such as Argyris triangle~\cite{ArgyrisFriedScharpf:1968} seems a natural choice for discretization. The performance of Argyris triangle and several other C$^1$-conforming finite elements has been carefully studied in~\cite{Fisher:2010} for a nonlinear strain gradient elastic model. A drawback of the C$^1$-conforming elements is that the number of the degrees of freedom is large and high order polynomial has to be used in the shape functions, which is more pronounced for three dimensional problems; See, e.g., the finite element for a three-dimensional strain gradient model proposed in~\cite{Zerovs:20092} locally has $192$ degrees of freedom. We aim to develop some simple and robust nonconforming elements to avoid such difficulties. The robustness is understood in the sense that the elements converge uniformly in the energy norm with respect to $\iota$.

To this end, we firstly prove a new H$^2-$Korn's inequality and its discrete analog in any dimension. This H$^2-$Korn's inequality may be viewed as a quantitative version of the so-called {\em vector version of J.L. Lions lemma}~\cite[Theorem 6.19-1]{Ciarlet:2013}; See also~\eqref{eq:lions}, while our proof is constructive and may be adapted to prove a Korn's inequality for piecewise H$^2$ vector field (broken H$^2-$Korn's inequality for short), which may be viewed as a higher-order counterpart of {\sc Brenner's} seminal Korn's inequality~\cite{Brenner:2004} for piecewise H$^1$ vector fields. Compared to the broken H$^2-$inequality proved in~\cite{LiMingShi:2017}, the jump term associated with the gradient tensor of the piecewise vector field may be dropped. Therefore, the degrees of freedom associated with the gradient tensor along each face or edge may be dropped, which simplify the construction of the elements. Based on this observation, all H$^1$ conforming but H$^2$ nonconforming elements are suitable for approximating this strain gradient model. We choose the Specht triangle~\cite{Specht:1988} and the NZT tetrahedron~\cite{WangShiXu:2007} as two typical representatives. The Specht triangle is simpler than those in~\cite{LiMingShi:2017}, because the elements therein locally belong to a $21$ dimensional subspace of quintic polynomials, while the tensor products of the Specht triangle locally belong to an $18$ dimensional subspace of quartic polynomials. It is worth mentioning that the broken H$^2-$Korn's inequality may also be exploited to develop C$^0$ interior penalty method~\cite{Engel:2002, Brenner:2011} for the strain gradient elastic model.

To prove the robustness of both elements, we construct a regularized interpolation operator and an enriching operator, and derive certain estimates for such operators, which are key to prove sharp error estimate for problems with less smooth solution. These two operators are also useful for strain gradient elasticity model with other type boundary conditions.

The remaining part of the paper is organized as follows. We prove the continuous and the broken H$^2-$Korn's inequalities in \S 2. The Specht triangle and the NTZ tetrahedron are introduced in \S 3 and the corresponding regularized interpolant are constructed and analyzed therein. We introduce enriching operators for both elements in \S 4, and derive the error bounds uniformly with respect to $\iota$ in the same part. The numerical tests of both elements are reported in the last Section, which confirm the theoretical prediction in \S 4.

Throughout this paper, the constant $C$ may differ from line to line, while it is independent of the mesh size $h$ and the materials parameter $\iota$.
\section{H$^2-$Korn's Inequalities}
In this part we prove the H$^2-$Korn's inequalities and the broken H$^2-$Korn's inequalities. Let us fix some notations firstly.
\subsection{Notations}
Let $\Om\subset\mb{R}^d$ be a bounded convex polytope. We shall use the standard notations for Sobolev spaces, norms and semi-norms~\cite{AdamsFournier:2003}. The function space $L^2(\Om)$ consists functions that are square integrable over $\Om$, which is equipped with norm $\nm{\cdot}{L^2(\Om)}$ and the inner product $(\cdot,\cdot)$. Let $H^m(\Om)$ be the Sobolev space of square integrable functions whose weak derivatives up to order $m$ are also square integrable, the corresponding norm
\(
\nm{v}{H^m(\Om)}^2{:}=\sum_{k=0}^m\abs{v}_{H^k(\Om)}^2
\)
with the semi-norm $\abs{v}_{H^k(\Om)}^2{:}=\sum_{\abs{\alpha}=k}\nm{\pa^{\al}v}{L^2(\Om)}^2$.

For a positive number $s$ that is not an integer, $H^s(\Om)$ is the fractional order Sobolev space. Let $m=\lfloor\,s\rfloor$ be the largest integer less than $s$ and $\varrho=s-m$. The sem-inorm $\abs{v}_{H^s(\Om)}$ and the norm $\nm{v}{H^s(\Om)}$ are given by
\begin{align*}
&\abs{v}_{H^s(\Om)}^2=\sum_{\abs{\al}=m}\int_{\Om}\int_{\Om}
\dfrac{\abs{(\pa^{\al})v(x)-(\pa^{\al})v(y)}^2}{\abs{x-y}^{2+2\varrho}}
\dx\,\dy,\\
&\nm{v}{H^s(\Om)}^2=\nm{v}{H^m(\Om)}^2+\abs{v}_{H^s(\Om)}^2.
\end{align*}

By~\cite[\S 7]{AdamsFournier:2003}, the above definition for the fractional order Sobolev space $H^s(\Om)$ is equivalent to the one obtained by interpolation, i.e.,
\[
H^s(\Om)=\left[H^{m+1}(\Om),H^m(\Om)\right]_{\theta}\qquad\text{with\quad}\theta=m+1-s.
\]
In particular, there exists $C$ that depends on $\Om$ and $s$ such that
\begin{equation}\label{eq:interineq}
\nm{v}{H^s(\Om)}\le C\nm{v}{H^{m+1}(\Om)}^{1-\theta}\nm{v}{H^m(\Om)}^{\theta}.
\end{equation}
For $s\ge 0, H_0^s(\Om)$ is the closure in $H^s(\Om)$ of the space of $C^\infty(\Om)$ functions with compact supports in $\Om$. In particular,
\begin{align*}
&H^1_0(\Om){:}=\set{v\in H^1(\Om)}{v=0\text{\;on\;}\pa\Om},\\
&H^2_0(\Om){:}=\set{v\in H^2(\Om)}{v=\pa_n v=0\text{\;on\;}\pa\Om},
\end{align*}
where $\pa_n v$ is normal derivative of $v$.

For any vector-valued function $v$, its gradient $\na v$ is a matrix-valued
function given by $(\na v)_{ij}=\pa_i v_j$ for $i,j=1,\cdots, d$. The strain tensor $\eps(v)$ is given by $\eps(v)=\frac{1}{2}(\na v+[\na v]^T)$ with $\eps_{ij}=\frac{1}{2}\lr{\pa_iv_j+\pa_jv_i}$. The divergence operator is defined by $\divop v=\sum_{i=1}^d\pa_iv_i$. The spaces $[H^m(\Om)]^d, [H_0^m(\Om)]^d$ and $[L^2(\Om)]^d$ are standard Sobolev spaces for the vector fields. Without abuse of notation, we employ $\abs{\cdot}$ to denote the abstract value of a scalar, the $\ell_2$ norm of a vector, and the Euclidean norm of a matrix. Throughout this paper, we may drop the subscript $\Om$ whenever no confusion occurs.

Let $\mc{T}_h$ be a simplicial triangulation of $\Omega$ with maximum mesh size $h$. We assume all elements in $\mc{T}_h$ are shape-regular in the sense of Ciarlet and Raviart~\cite{Ciarlet:1978}, i.e., there exists a constant $\gamma$ such that $h_K/\rho_K\le\gamma$, where $h_K$ is the diameter of element $K$, and $\rho_K$ is the diameter of the largest ball inscribed into $K$, and $\gamma$ is the so-called chunkiness parameter ~\cite{BrennerScott:2008}. We denote by $\mc{F}_h$, $\mc{E}_h$ and $\mc{V}_h$ the sets of $(d-1)-$dimensional faces, edges and vertices, respectively. Let $\mc{F}_h^{\,B}=\set{f\in\mc{F}_h}{f\subset\pa\Om}$ be the set of boundary faces. We denote by $\mc{F}_h^{\,I}=\mc{F}_h\setminus\mc{F}_h^{\,B}$ the set of interior faces. Similar notations apply to $\mc{E}_h$ and $\mc{V}_h$. We denote by $\mc{T}_h^{\,B}=\set{K\in\mc{T}_h}{\exists f\in\mc{F}_h^{\,B},\,\text{s.t.}\,f\subset\pa K}$ the set of boundary simplexes. Given a simplex or sub-simplex $B$, we let $\mc{T}_h(B)=\set{K\in\mc{T}_h}{B\subset K}$ be the element star containing $B$. Similar notations apply to the set of faces, edges and vertices. For example, for any node $a\in\mc{V}_h^{\,B}$, the $\mc{E}_h^{\,B}(a)$ is the set of boundary edges sharing a common node $a$.

We classify the boundary vertices as follows. We say that a node $a\in\mc{V}_h^B$ is a {\em flat node} if the boundary edges set $\mc{E}_h^{\,B}(a)$ span a $(d-1)$-dimensional linear space. Otherwise, we say that $a$ is a {\em sharp node}. We let $\mc{V}_h^B=\mc{V}_h^{\,\flat}\cup\mc{V}_h^{\;\#}$, where $\mc{V}_h^{\,\flat}$ and $\mc{V}_h^{\;\#}$ denote the sets of the flat node and sharp node, respectively. For any sharp node $a\in\mc{V}_h^{\,\#}$, we may choose a set with $d$ linear independent boundary edges, i.e., $\mc{E}_a=\set{e_i\in\mc{V}_h^{\,B}(a)}{1\le i\le d}$, such that $\mc{E}_a$ provide a basis of $\mb{R}^d$.
\subsection{H$^2$-Korn's inequality}
We write the boundary value problem~\eqref{eq:sgbvp} into the following variational problem: Find $u\in[H^2_0(\Om)]^d$ such that
\begin{equation}\label{variation:eq}
a(u,v)=(f,v)\quad\text{for all\quad} v\in [H_0^2(\Om)]^d,
\end{equation}
where the bilinear form $a$ is defined for any $v,w\in [H_0^2(\Om)]^d$ as
\[
a(v,w){:}=(\C\eps(v),\eps(w))+\iota^2(\D\na\eps(v),\na\eps(w)),
\]
and the fourth-order tensors $\C$ and the sixth-order tensor $\D$ are defined by
\[
\C_{ijkl}=\lam\del_{ij}\del_{kl}+2\mu\del_{ik}\del_{jl}\quad\text{and}\quad
\D_{ijklmn}=\lam\del_{il}\del_{jk}\del_{mn}+2\mu\del_{il}\del_{jm}\del_{kn},
\]
respectively. Here $\delta_{ij}$ is the Kronecker delta function. The strain gradient $\na\eps(v)$ is a third order tensor defined by $(\na\eps(v))_{ijk}=\eps_{jk,i}$.

The wellposedness of Problem~\eqref{variation:eq} depends on the coercivity of the bilinear form $a$ over $[H_0^2(\Om)]^d$, which is a direct consequence of the following H$^2-$Korn's inequality:
\begin{equation}\label{eq:h2korn}
\nm{\eps(v)}{L^2}^2+\nm{\na\eps(v)}{L^2}^2\ge C(\Om)\nm{\na v}{H^1}^2\qquad\text{for all\quad}v\in[H_0^2(\Om)]^d.
\end{equation}
This inequality was proved in~\cite[Theorem 1]{LiMingShi:2017} by exploiting the community property of the strain operator $\eps$ and the partial derivative operator $\pa$, which has been implicitly used in~\cite{AbelMoraMuller:2011} to prove an inequality similar to~\eqref{eq:h2korn}.

In what follows, we prove that~\eqref{eq:h2korn} remains valid for a more general vector field $v$ in $[H^2(\Om)\cap H_0^1(\Om)]^d$. The precise form is stated in~\eqref{eq:korn}. Our proof relies on the fact that the strain gradient field fully controls the Hessian of the displacement algebraically; See, cf.~\eqref{eq:algkorn}. This fact will be further exploited to prove a discrete analog of~\eqref{eq:korn} for a piecewise vector field, which is dubbed as the broken H$^2-$Korn's inequality. We have exploited a weaker version of such broken H$^2-$Korn's inequality to design two robust strain gradient finite elements in~\cite{LiMingShi:2017}. 
\begin{theorem}\label{lema:h2korn}
For any $v\in [H_0^1(\Om)]^d$ and $\na\eps(v)\in [L^2(\Om)]^{d\x d\x d}$, there holds $\na v\in [H^1(\Om)]^{d\x d}$ and
\begin{equation}\label{eq:korn}
\nm{\eps(v)}{L^2}^2+\nm{\na\eps(v)}{L^2}^2\ge \dfrac14\Lr{\nm{\na v}{L^2}^2+\nm{D^2 v}{L^2}^2}.
\end{equation}
Here, the constant $1/4$ in right hand side of the above inequality may be replaced by $1-1/\sqrt{2}$ when $d=2$.
\end{theorem}
\begin{proof}
The core of the proof is the following algebraic inequality:
\begin{equation}\label{eq:algkorn}
\abs{\na\eps(v)}^2\ge \frac{1}{4}\abs{\na^2 v}^2.
\end{equation}
Integrating~\eqref{eq:algkorn} over domain $\Om$, we obtain
\begin{equation}\label{eq:2ndder}
\nm{\na\eps(v)}{L^2}^2\ge\frac{1}{4}\nm{\na^2 v}{L^2}^2,
\end{equation}
which together with the first Korn's inequality~\cite{Korn:1908, Korn:1909}:
\begin{equation}\label{eq:1stkorn}
2\nm{\eps(v)}{L^2}^2\ge\nm{\na v}{L^2}^2\qquad\text{for all\quad}v\in [H_0^1(\Om)]^d
\end{equation}
implies~\eqref{eq:korn}.

To prove~\eqref{eq:algkorn}, we start with the identity
\begin{equation}\label{eq:defsg}
\begin{aligned}
\abs{\na\eps(v)}^2&=\sum_{1\leq i,j,k\leq d}\abs{\pa_i\eps_{jk}}^2\\
&=\sum_{1=1}^d\abs{\pa_i\eps_{ii}}^2+\sum_{1\le i<j\le d}\lr{\abs{\pa_i\eps_{jj}}^2+2\abs{\pa_j\eps_{ij}}^2}+\lr{\abs{\pa_j\eps_{ii}}^2+2\abs{\pa_i\eps_{ij}}^2}\\
&\quad+\sum_{i\neq j\neq k}\abs{\pa_i\eps_{jk}}^2{:}=I_1+I_2+I_3,
\end{aligned}
\end{equation}
where $I_3$ vanishes for $d=2$.

Employing the elementary algebraic inequality
\[
a^2+\dfrac12(a+b)^2\ge(1-1/\sqrt2)(a^2+b^2), \qquad a,b\in\R,
\]
we obtain
\begin{equation}\label{ieq:1}
\left\{
\begin{aligned}
\abs{\pa_{j}\eps_{ii}}^2+2\abs{\pa_i\eps_{ij}}^2&\ge(1-1/\sqrt2)\Lr{\abs{\pa_{ij}v_i}^2+\abs{\pa_{ii}v_j}^2},\\
\abs{\pa_{i}\eps_{jj}}^2+2\abs{\pa_j\eps_{ij}}^2&\ge(1-1/\sqrt2)\Lr{\abs{\pa_{ij}v_j}^2+\abs{\pa_{jj}v_i}^2}.
\end{aligned}\right.
\end{equation}
Using the elementary algebraic equality
\[
\sum_{i=1}^d(a_i+a_{i+1})^2=\sum_{i=1}^da_i^2+\Lr{\sum_{i=1}^da_i}^2\qquad\text{for\quad}a_i\in\R\text{\quad and\quad}a_{d+1}=a_1,
\]
we obtain
\begin{equation}\label{ieq:2}
\sum_{1\leq i\neq j\neq k\leq d}\abs{\pa_i\eps_{jk}}^2=\dfrac14\sum_{1\leq i\neq j\neq k\leq d}\abs{\pa_{ij}v_k}^2+\dfrac14\Lr{\sum_{1\leq i\neq j\neq k\le d}\pa_{ij}v_k}^2.
\end{equation}
Combining ~\eqref{eq:defsg}, ~\eqref{ieq:1} and ~\eqref{ieq:2}, we obtain~\eqref{eq:algkorn} immediately.

$I_3$ vanishes when $d=2$. Therefore, the constant in the right-hand side of~\eqref{eq:korn} may be replaced by $1-1/\sqrt{2}$. This completes the proof.
\end{proof}

%
A direct consequence of Theorem~\ref{lema:h2korn} is the following full H$^2-$Korn's inequality.
\begin{coro}
Let $\Om\subset\R^d$ be a domain such that the following Korn's inequality is valid for
any vector field $v\in [L^2(\Om)]^d$ and $\eps(v)\in[L^2(\Om)]^{d\x d}$,
\[
\nm{v}{L^2}+\nm{\eps(v)}{L^2}\ge C(\Om)\nm{v}{H^1}.
\]
If $v\in [L^2(\Om)]^d,\eps(v)\in[L^2(\Om)]^{d\x d}$ and $\na\eps(v)\in [L^2(\Om)]^{d\x d\x d}$,
then $v\in [H^2(\Om)]^d$ and
\begin{equation}\label{eq:fullkorn}
\nm{v}{L^2}+\nm{\eps(v)}{L^2}+\nm{\na\eps(v)}{L^2}\ge\min\Lr{C(\Om),1/2}\nm{v}{H^2}.
\end{equation}
\end{coro}

In~\cite[Theorem 6.19-1]{Ciarlet:2013}, the following vector version of {\em J.L. Lions Lemma} is proved: For any domain $D$ in $\mb{R}^d$ and $m\in\mb{Z}$, then
\begin{equation}\label{eq:lions}
v\in [H^m(D)]^d\quad\text{and}\quad\eps(v)\in[H^m(D)]^{d\x d}\text{\quad implies\quad}v\in [H^{m+1}(D)]^d.
\end{equation}
The inequality~\eqref{eq:fullkorn} may be viewed as a quantitative version of~\eqref{eq:lions}
with $m=1$, while the proof in~\cite[Theorem 6.19-1]{Ciarlet:2013} is nonconstructive and is not easy to be adapted to prove the broken H$^2-$Korn's inequality for a piecewise vector filed.

The regularity of Problem~\eqref{variation:eq} is essential to prove a uniform error estimate. Unfortunately, it does not seem easy to identify such estimates in the literature, and we give a proof for the readers' convenience. We firstly make an extra regularity assumption.
\begin{hypo}\label{hypo}
Let $u$ be a solution of following equations,
\[
\left\{
\begin{aligned}
&\triangle\lr{\mc{L}u}=f\quad&&\text{in\;}\Om,\\
&u=\pa_nu=0\quad&&\text{on\;}\pa\Om,
\end{aligned}\right.
\]
where $\mc{L}u{:}=\mu\triangle u+(\lambda+\mu)\na\divop u$. Then there holds that for any $f\in H^{-1}(\Om)$,
\begin{equation}\label{hypo:regularity}
\nm{u}{H^3}\le C\nm{f}{H^{-1}}.
\end{equation}
\end{hypo}

If $\Om$ is a smooth domain, the regularity property~\eqref{hypo:regularity} is standard; See e.g.,~\cite{Agmon:1964}. While it is unclear whether the above regularity estimate is true for a convex polytope. Nevertheless, if the operator $\mc{L}$ is replaced by the Laplacian operator,  then~\eqref{hypo:regularity} is proved in~\cite[Chapter 4 Theorem 4.3.10]{Mazya:2010}.
\begin{lemma}\label{lemma:reg}
Assume Hypothesis~\ref{hypo} is valid and let $u$ be the solution of~\eqref{variation:eq}, then there exists $C$ that may depend on $\Om$ but independent of $\iota$ such that
\begin{equation}\label{eq:regularity}
\nm{\na^k(u-u_0)}{L^2}\le C\iota^{3/2-k}\nm{f}{L^2}\quad\text{for\quad}k=1,2,
\end{equation}
where $u_0\in [H_0^1(\Om)]^d$ satisfies
\begin{equation}\label{eq:limit}
(\C\eps(u_0),\eps(v))=(f,v)\qquad \text{for all\quad}v\in[H_0^1(\Om)]^d.
\end{equation}

Moreover, we have
\begin{equation}\label{eq:estfrac1}
\nm{u}{H^{3/2}}\le C\nm{f}{L^2},
\end{equation}
and
\begin{equation}\label{eq:regfrac2}
\nm{u}{H^{5/2}}\le C\iota^{-1}\nm{f}{L^2}.
\end{equation}
\end{lemma}

Under Hypothesis~\ref{hypo}, we may prove this regularity result by following essentially the same line of the proof in~\cite[Lemma 5.1]{Tai:2001}. We include it here for completeness.
\begin{proof}
By~\eqref{eq:sgbvp} and~\eqref{eq:limit}, we have
\[
\triangle\mc{L}(u)=\iota^{-2}\mc{L}(u-u_0).
\]
Using the regularity hypothesis ~\eqref{hypo:regularity}, we obtain
\begin{equation}\label{regularity:H3}
\nm{u}{H^3}\le C\iota^{-2}\nm{\mc{L}(u-u_0)}{-1}\le C\iota^{-2}\lr{\mb{C}\eps(u-u_0),\eps(u-u_0)}^{\frac{1}{2}}.
\end{equation}
By the standard regularity estimate, we have
\begin{equation}\label{estimate:u0}
\nm{u_0}{H^2}\le C\nm{f}{L^2}.
\end{equation}

Denoting $\phi=u-u_0$ and integration by parts, we have
\begin{align*}
\iota^2\lr{\mb{D}\na\eps(\phi),\na\eps(\phi)}+\lr{\mb{C}\eps(\phi),\eps(\phi)}&=-\iota^2\lr{\mb{D}\na\eps(u_0),\na\eps(\phi)}\\
&\quad+\iota^2\int_{\pa\Om}M_{nn}(u)\pa_n u_0\,\md\sigma(x),
\end{align*}
where $M_{nn}(u)=n^{\mathrm{T}}\cdot\mb{D}\na\eps(u)\cdot n$.

Using the regularity estimate~\eqref{estimate:u0}, we obtain
\[
\iota^2\lr{\mb{D}\na\eps(u_0),\na\eps(\phi)}\le\frac{\iota^2}{2}\lr{\mb{D}\na\eps(\phi),\na\eps(\phi)}+\dfrac{\iota^2}{2}\nm{f}{L^2}^2.
\]

Using the trace inequality~\eqref{eq:tracedomain}, we obtain, for $\del>0$ to be chosen later,
\begin{align*}
\iota^2\left\lvert\int_{\pa\Om}M_{nn}(u)\pa_nu_0\md\sigma(x)\right\rvert&\le\iota^3\delta\nm{M_{nn}(u)}{L^2(\pa\Om)}^2+\dfrac{\iota}{4\delta}\nm{\pa_nu_0}{L^2(\pa\Om)}^2\\
&\le C\delta\Lr{\iota^4\nm{\na^2u}{H^1}^2+\iota^2\nm{\na^2 u}{L^2}^2}+\dfrac{C}{\delta}\iota\nm{u_0}{H^2}^2.
\end{align*}
Using~\eqref{eq:2ndder}, we obtain
\[
\nm{\na^2 u}{L^2}^2\le 2\nm{\na^2\phi}{L^2}^2+2\nm{\na^2 u_0}{L^2}^2\le\dfrac{4}{\mu}\lr{\mb{D}\na\eps(\phi),\na\eps(\phi)}
+2\nm{\na^2 u_0}{L^2}^2.
\]
Using the regularity estimates~\eqref{regularity:H3} and~\eqref{estimate:u0}, we bound the right-hand side of the above inequality as
\begin{align*}
\iota^2\labs{\int_{\pa\Om}M_{nn}(u)\pa_nu_0\md\sigma(x)}
&\le C\delta\bigl[\lr{\mb{C}\eps(\phi),\eps(\phi)}+\iota^2\lr{\mb{D}\na\eps(\phi),\na\eps(\phi)}\bigr]\\
&\quad+C\iota\Lr{\delta\iota+\dfrac1{\delta}}\nm{f}{L^2}^2.
\end{align*}

Combining the above inequalities,  we obtain
\begin{align*}
\iota^2\lr{\mb{D}\na\eps(\phi),\na\eps(\phi)}+\lr{\mb{C}\eps(\phi),\eps(\phi)}&\le C\delta\bigl[\lr{\mb{C}\eps(\phi),\eps(\phi)}+\iota^2\lr{\mb{D}\na\eps(\phi),\na\eps(\phi)}\bigr]\\
&\quad+\dfrac{\iota^2}{2}
\lr{\mb{D}\na\eps(\phi),\na\eps(\phi)}
+C\iota^2\delta\nm{f}{L^2}^2+\dfrac{C\iota}{\delta}\nm{f}{L^2}^2.
\end{align*}
Choosing $\del$ properly, we obtain~\eqref{eq:regularity}.

Using~\eqref{eq:regularity} and Poincar\'{e} inequality, and noting that $\iota<1$, we obtain
\begin{equation}\label{eq:limiterr2}
\nm{u-u_0}{H^2}\le C\Lr{\iota^{1/2}+\iota^{-1/2}}\nm{f}{L^2}\le C\iota^{-1/2}\nm{f}{L^2},
\end{equation}
and
\begin{equation}\label{eq:limiterr3}
\nm{u}{H^2}\le\nm{u-u_0}{H^2}+\nm{u_0}{H^2}\le C\iota^{-1/2}\nm{f}{L^2}.
\end{equation}
Interpolating~\eqref{eq:limiterr2} and~\eqref{eq:regularity} with $k=1$, we obtain
\[
\nm{u-u_0}{H^{3/2}}\le C\nm{f}{L^2}.
\]
Using the interpolation inequality~\eqref{eq:interineq}, we obtain
\[
\nm{u_0}{H^{3/2}}\le C\nm{u_0}{H^1}^{1/2}\nm{u_0}{H^2}^{1/2}\le C\nm{f}{L^2}.
\]
A combination of the above two inequalities yields~\eqref{eq:estfrac1}.

Combining~\eqref{regularity:H3} and~\eqref{eq:regularity}, we obtain
\[
\nm{u}{H^3}\le C\iota^{-2}\nm{\na(u-u_0)}{L^2}\le C\iota^{-3/2}\nm{f}{L^2}.
\]
Interpolating the above inequality and~\eqref{eq:limiterr3}, we obtain~\eqref{eq:regfrac2}. 
\end{proof}
\subsection{The broken H$^2$-Korn's inequality}
For any $m\in\mb{N}$, the space of piecewise vector fields is defined by
\[
[H^m(\Om,\mc{T}_h)]^d{:}=\set{v\in [L^2(\Om)]^d}{v|_K\in[H^m(K)]^d\quad\text{for all\quad} K\in\mc{T}_h},
\]
which is equipped with the broken norm
\[
\nm{v}{H_h^k}{:}=\nm{v}{L^2}+\sum_{k=1}^m\nm{\na^k_h v}{L^2},
\]
where
\(
\nm{\na^k_h v}{L^2}^2=\sum_{K\in\mc{T}_h}\nm{\na^k v}{L^2(K)}^2
\)
with $(\na^k_h v)|_K=\na^k(v|_K)$. Moreover, $\eps_h(v)=(\na_h v+[\na_h v]^T)/2$. For any $v\in H^m(\Om,\mc{T}_h)$, we denote by $\jump{v}$ the jump of $v$ across the faces or the edge.

The main result of this part is the following broken H$^2$-Korn's inequality.
\begin{theorem}\label{thm:hkorn}
For any $v\in[H^2(\Omega,\mc{T}_h)]^d$, there exits $C$ that depends on $\Om$ and $\gamma$ but independent of $h$ such that
\begin{equation}\label{eq:diskorn}
\begin{aligned}
\nm{v}{H_h^2}^2\le C\biggl(\nm{\na_h\eps_h(v)}{L^2}^2&+\nm{\eps_h(v)}{L^2}^2+\nm{v}{L^2}^2\\
&\quad+\sum_{f\in\mc{F}_h}h_{e}^{-1}\nm{\jump{\Pi_{f}v}}{L^2(f)}^2\biggr),
\end{aligned}
\end{equation}
where $\Pi_f:[L^2(f)]^d\mapsto [P_{1,-}(f)]^d$ is the $L^2$ projection and
\[
[P_{1,-}(f)]^d{:}=\set{v\in [P_1(f)]^d}{v\cdot t\in\text{RM}(F)},
\]
where $t$ is the tangential vector of the face $f$ (or edge for $d=2$), and $\text{RM}(f)$ is the infinitesimal rigid motion on $f$.
\end{theorem}

For a piecewise vector filed $v$, the inequality~\eqref{eq:diskorn} improves the one proved in~\cite[Theorem 2]{LiMingShi:2017} by removing the jump term
\[
\sum_{i=1}^2\sum_{f\in\mc{F}_h}h_f^{-1}\nm{\jump{\Pi_{f}(v_{,i})}}{L^2(f)}^2.
\]
This term stands for the jump of the gradient tensor of the vector field $v$ across the element boundary. This would simplify the construction of the robust strain gradient elements as shown in the next two parts.
\vskip .5cm
\noindent{\em Proof of Theorem~\ref{thm:hkorn}\;}
Integrating~\eqref{eq:algkorn} over element $K\in\mc{T}_h$, we obtain,
\[
\nm{\na\eps(v)}{L^2(K)}^2\ge\frac{1}{4}\nm{\na^2 v}{L^2(K)}^2.
\]
Summing up all $K\in\mc{T}_h$, we get
\begin{equation}\label{eq:sghessian}
\nm{\na_h\eps_h(v)}{L^2}\ge\frac{1}{4}\nm{\na_h^2 v}{L^2}^2,
\end{equation}
which together with the following Korn's inequality for a piecewise H$^1$ vector filed proved by {\sc Mardal and Winther}~\cite{MardalWinther:2006}
\[
\nm{v}{H_h^1}^2\le C\Lr{\|\eps_h(v)\|_{L^2}^2+\nm{v}{L^2}^2
+\sum_{f\in\mc{F}_h}h_f^{-1}\nm{\jump{\Pi_fv}}{L^2(f)}^2}
\]
implies~\eqref{eq:diskorn}.
\qed

We shall frequently use the following trace inequalities.
\begin{lemma}
For any Lipschitz domain $D$, there exists $C$ depending on $D$ such that
\begin{equation}\label{eq:tracedomain}
\nm{v}{L^2(\pa D)}\le C\nm{v}{L^2(D)}^{1/2}\nm{v}{H^1(D)}^{1/2}.
\end{equation}

For an element $K$, there exists $C$ independent of $h_K$, but depends on $\gamma$ such that
\begin{equation}\label{eq:eletrace}
\nm{v}{L^2(\pa K)}\le C\Lr{h_K^{-1/2}\nm{v}{L^2(K)}+h_K^{1/2}\nm{\na v}{L^2(K)}}.
\end{equation}

If $v\in\mb{P}_m(K)$, then there exists $C$ independent of $v$, but depends on $\gamma$ and $m$ such that
\begin{equation}\label{eq:polytrace}
\nm{v}{L^2(\pa K)}\le Ch_K^{-1/2}\nm{v}{L^2(K)}.
\end{equation}
\end{lemma}

The multiplicative type trace inequality~\eqref{eq:tracedomain} may be found in~\cite[Theorem 1.5.1.10]{Grisvard:1985}, while~\eqref{eq:eletrace} is a direct consequence of~\eqref{eq:tracedomain}. The third trace inequality is a combination of~\eqref{eq:eletrace} and the inverse inequality for any polynomial $v\in\mb{P}_m(K)$.
\section{Interpolation for nonsmooth data}
Motivated by the broken H$^2-$Korn's inequality~\eqref{eq:diskorn}, we conclude that the H$^1-$conforming but H$^2-$nonconforming finite elements are natural choices for approximating Problem~\eqref{eq:disvara}. A family of rectangular elements in this vein may be found in~\cite{LiaoMing:2019}, and two nonconforming tetrahedron elements were constructed and analyzed in~\cite{Wang:2020}. Note that the tensor product of certain finite elements for the singular perturbation problem of fourth order may also be used to approximate~\eqref{eq:sgbvp}, we refer to~\cite{Semper:1992, Semper:1994, GuzmanLeykekhmanNeilan:2012, Tai:2001, Tai:2006} and references therein for such elements.  In what follows, we select the Specht triangle~\cite{Specht:1988} and the NZT tetrahedron~\cite{WangShiXu:2007} as the representatives. The Specht triangle is a successful plate bending element,~\textit{which passes all the patch tests and performs excellently, and is one of the best thin plate triangles with $9$ degrees of freedom that currently available}~\cite[Quatation in p. 345]{ZienkiewiczTaylor:2009}. The NZT tetrahedron may be regarded as a three-dimensional extension of the Specht triangle.

The Specht triangle and the NZT tetrahedron may be defined by the finite element triple $(K,P_K,\Sigma_K)$~\cite{Ciarlet:1978} in a unifying way as following: $K$ is a simplex, and
\[
\left\{\begin{aligned}
P_K&=Z_K+b_K\mb{P}_1(K),\\
\Sigma_K&=\{p(a_i),(e_{ij}\cdot\na p)(a_i),1\leq i\neq j\leq d+1\}
\end{aligned}\right.
\]
with extra constraints
\begin{equation}\label{eq:spechtconstraint}
\frac{1}{\abs{f_i}}\int_{f_i}\pa_n p=\frac{1}{d}\sum_{1\leq k\leq d+1,k\neq i}\pa_n p(a_k),\quad i=1,\cdots,d+1,
\end{equation}
where $f_i$ is a $(d-1)$ dimensional simplex opposite to vertex $a_i$, and $e_{ij}$ is the edge vector from $a_i$ to $a_j$. Here $Z_K$ is the Zienkiewicz space defined by
\[
Z_K=\mb{P}_2(K)+\text{Span}\set{\lam_i^2\lam_j-\lam_i\lam_j^2}{1\le i\not= j\le d+1},
\]
where $\lam_i$ is the barycentric coordinates associated with the vertex $a_i$.

The finite element space is define by
\[
X_h{:}=\set{v\in H^1(\Om)}{v|_K\in P_K, K\in\mc{T}_h;v(a),\na v(a)\,\text{are continuous for}\,a\in\mc{V}_h}.
\]
The corresponding homogenous finite element space is defined by
\[
X_h^{\,0}{:}=\set{v\in X_h}{v(a),\na v(a)\,\text{vanish on $\pa\Om$ for}\,a\in\mc{V}_h}.
\]
It is clear that $X_h^{\,0}\subset H_0^1(\Om)$. We denote $V_h=[X_h^{\,0}]^d$, and approximating problem reads as: Find $u_h\in V_h$ such that
\begin{equation}\label{eq:disvara}
a_h(u_h,v)=(f,v)\quad\text{for all\quad} v\in V_h,
\end{equation}
where the bilinear form $a_h$ is defined for any $v,w\in V_h$ as
\[
a_h(v,w){:}=(\C\eps(v),\eps(w))+\iota^2(\D\na_h\eps(v),\na_h\eps(w))
\]
with
\[
(\D\na_h\eps(v),\na_h\eps(w)){:}=\sum_{K\in\mc{T}_h}\int_K\D\na\eps(v)\na\eps(w)\dx.
\]
The energy norm is defined as $\wnm{v}{:}=\Lr{\nm{v}{H^1}^2+\iota^2\nm{\na_h^2 v}{L^2}^2}^{1/2}$. The bilinear form is coercive in this energy norm as shown in the next lemma.
\begin{lemma}\label{lema:coer1}
For any $v\in V_h$,
\begin{equation}\label{eq:coer1}
a_h(v,v)\ge\frac{\mu}{2+2\abs{C_p}^2}\wnm{v}^2,
\end{equation}
where $C_p$ appears in the Poincar\'e inequality
\begin{equation}\label{eq:poincare}
\nm{v}{L^2}\le C_p\nm{\na v}{L^2}\qquad\text{for all\quad}v\in [H_0^1(\Om)]^d.
\end{equation}
\end{lemma}
The estimate~\eqref{eq:coer1} immediately implies the wellposedness of Problem~\eqref{eq:disvara} for any fixed $\iota$.
\begin{proof}
For any $v\in V_h$, there holds
\[
a_h(v,v)\ge 2\mu\Lr{\nm{\eps(v)}{L^2}^2+\iota^2\nm{\na_h\eps(v)}{L^2}^2}.
\]
Using the first Korn's inequality~\eqref{eq:1stkorn} and the estimate~\eqref{eq:sghessian}, we obtain
\[
a_h(v,v)\ge\frac{\mu}{2}\Lr{\nm{\na v}{L^2}^2+\iota^2\nm{\na_h^2 v}{L^2}^2}.
\]
The coercivity estimate~\eqref{eq:coer1} follows by using the Poincare's inequality~\eqref{eq:poincare}.
\end{proof}

The standard interpolation estimate for the above elements reads as~\cite{CiarletRaviart:1972},
\[
\sum_{j=0}^3h^j\nm{\na^j_h(v-\vpi v)}{L^2}\le Ch^3\nm{\na^3 v}{L^2},
\]
This interpolant is unbounded in $H^2(\Om)$, which is even not well-defined for a function in $H^2(\Om)$. 

Our definition for the regularized interpolant is a combination of a regularized interpolant in~\cite{GuzmanLeykekhmanNeilan:2012} and an enriching operator defined in~\cite{Neilan:2019}.

Define $I_h{:}=\vpi_h\circ\vpi_C$ with $\vpi_C: H_0^1(\Om)\to L_h$ the Scott-Zhang interpolant~\cite{ScottZhang:1990}, where $L_h$ is the quadratic Lagrangian finite element space with vanishing trace. The operator $\vpi_h: L_h\to X_h$ is locally defined as follows.
\begin{enumerate}
\item If $a\in\mc{V}_h^{\,I}$ is an interior vertex, then we fix an element $K_a$ from $\mc{T}_h(a)$,
\[
\vpi_hw(a){:}=w(a)\quad\text{and\quad}\na\vpi_hw(a){:}=\na w_{|_{K_a}}(a).
\]

\item If $a\in\mc{V}_h^{\,\flat}$ is a flat node, then we fix an element $K_a$ from $\mc{T}_h^{\,B}(a)$, 
\[
\vpi_hw(a){:}=0\quad\text{and}\quad\na\vpi_hw(a){:}=\na w_{|_{K_a}}(a).
\]

\item If $a\in\mc{V}_h^{\,\#}$ is a sharp node, then 
\[
\vpi_hw(a){:}=0\quad\text{and}\quad\na\vpi_hw(a){:}=0.
\]
\end{enumerate}

\begin{remark}
For any $w\in H_0^1(\Om)\cap L_h$, there holds $\vpi_hw\in H_0^1(\Om)$.
\end{remark}

For any $v\in H_0^2(\Om)$, we define $I_h^{\,0}:H_0^2(\Om)\to X_h^{\,0}$ with
\[
I_h^{\,0}v(a)=I_hv(a),\,\na I_h^{\,0}v(a)=\na I_hv(a),\,\text{for all}\,\,a\in\mc{V}_h^{\,I}.
\]

\begin{theorem}\label{thm:reginter}
There exists an operator $I_h: H_0^1(\Om)\to X_h$ such that for any $v\in H^m(\Om)$ with $2\le m\le3$, there holds
\begin{equation}\label{interpolation:1}
\nm{v-I_h v}{H^k_h}\le Ch^{m-k}\abs{v}_{H^m},\quad 0\le k\le m.
\end{equation}

Moreover, there exists $I_h^0: H_0^2(\Om)\to X_h^{\,0}$ such that for any $v\in H^m(\Om)\cap H_0^2(\Om)$ with $2\le m\le 3$, there holds
\begin{equation}\label{interpolation:2}
\nm{v-I_h^{\,0} v}{H^k_h}\le Ch^{m-k}\abs{v}_{H^m},\quad 0\le k\le m.
\end{equation}
\end{theorem}

The interpolant $I_h^0 v$ is enough to our ends, while the more general interpolant $I_h v$ is a useful tool to deal with the strain gradient model with other boundary conditions; See, e.g.,~\cite{Aifantis:2011}. 
\begin{proof}
For any $\phi\in P_K$, a standard scaling argument yields that
\[
\nm{\phi}{L^2(K)}^2\le Ch_K^d\sum_{a\in\mc{V}_K}\Lr{\abs{\phi(a)}^2+h_K^2\abs{\na\phi(a)}^2},
\]
where $\mc{V}_K$ is the set containing all vertices of $K$. Let $w=\vpi_C v$ and $\phi=w-\vpi_h w$. Noting $\phi(a)=0$, we obtain
\begin{equation}\label{eq:scaling}
\nm{\phi}{L^2(K)}^2\le Ch_K^{d+2}\sum_{a\in\mc{V}_K}\labs{\na\phi(a)}^2.
\end{equation}

If the node $a\in\mc{V}_h^{\,I}$ or $a\in\mc{V}_h^{\,\flat}$, there has $\na\vpi_hw(a)=\lr{\na w_{|_{K_a}}}(a)$. Then we select a sequence of elements $\{K_1,\cdots,K_{J_a}\}\subset\mc{T}_h(a)$ such that $K_1=K,K_{J_a}=K_a$, and $f_j=K_j\cap K_{j+1}$ is a common $(d-1)$-dimensional simplex of $K_j$ and $K_{j+1}$. We write the right-hand side of the above inequality as the telescopic sum:
\[
\abs{\na\phi(a)}^2\le \sum_{j=1}^{J_a-1}\labs{\na w_{|_{K_j}}(a)-\na w_{|_{K_{j+1}}}(a)}^2.
\]
Note that $w$ is continuous across $f_j$, we write
\begin{equation}\label{eq:internode}
\abs{\na\phi(a)}^2\le C\sum_{f\in\mc{F}_a^{\,I}}h_F^{-2}\nm{\jump{\pa_{n_f}w}}{L^2(f)}^2.
\end{equation}

If the node $a\in\mc{V}_h^{\#}$, the edges $\mc{E}_a\subset\mc{E}_h^{\,B}(a)$ provide a basis of $\mb{R}^d$, we have
\[
\abs{\na\phi(a)}\le C\sum_{e\in\mc{E}_a}\abs{\pa_{e}w_{|_K}(a)-\pa_ew_{|_{K_e}}(a)},
\]
where $\pa_{e}w=e\cdot\na w$, $K_e$ is a simplex with the boundary edge $e$, and we employ the fact $w_{|_e}=0$. Proceeding along the same line by connecting $K$ to $K_e$ with a sequence of simplices, we obtain
\begin{equation}\label{eq:flatnode}
\abs{\na\phi(a)}^2\le C\sum_{f\in\mc{F}_a^{\,I}}\nm{\jump{\pa w/\pa_{n_f}}}{L^2(f)}^2.
\end{equation}

Substituting~\eqref{eq:internode},~\eqref{eq:flatnode} into~\eqref{eq:scaling}, we obtain
\[
\nm{\phi}{L^2(K)}\le Ch_K^{3/2}\sum_{a\in\mc{V}_K}\sum_{f\in\mc{F}_a^{\,I}}\nm{\jump{\pa_nw}}{L^2(f)},
\]
where we have used the fact that $h_K\simeq h_{f}$ because $\mc{T}_h$ is locally quasi-uniform.

Therefore, using the above inequality, the estimate for the Scott-Zhang interpolant, the inverse inequality and the trace inequality ~\eqref{eq:eletrace}, we obtain
\begin{align*}
\nm{v-I_hv}{H^k(K)}&\le\nm{v-\vpi_Cv}{H^k(K)}+\nm{\vpi_Cv-\vpi_h\vpi_Cv}{H^k(K)}\\
&\le Ch_K^{m-k}\nm{v}{H^m(K)}+Ch_K^{\frac{3}{2}-k}\sum_{a\in\mc{V}_K}\sum_{f\in\mc{F}_a^{\,I}}\nm{\pa_n(\vpi_C v-v)}{L^2(f)}\\
&\le Ch_K^{m-k}\nm{v}{H^m(\mc{T}_K)},
\end{align*}
where $\mc{T}_K=\cup_{a\in\mc{V}_K}\mc{T}_h(a)$ is the local element star of $K$. Summing up the above inequalities for $K\in\mc{T}_h$, we obtain~\eqref{interpolation:1}.

Next, for any $v\in H_0^2(\Om)\cap H^m(\Om)$, the estimate~\eqref{interpolation:2} may be proceeded along the same line that leads to~\eqref{interpolation:1}, we omit the details.
\end{proof}
%
%
\section{Error Estimate for Less Smooth Solution}
The standard error estimate argument is valid under the regularity assumption $u\in H^s(\Om)$ with $s\ge3$, which is usually invalid for the point loading or nonconvex domain~\cite{BlumRannacher:1980}. In this part we shall exploit enriching operator to derive a new error estimate for Problem~\eqref{eq:sgbvp} with less smooth solution. The enriching operator measures the distance between $V_h$ and $H^2(\Om)$, which was firstly introduced by {\sc Brenner}~\cite{Brenner:1994, Brenner:1996} to analyze nonconforming elements in the context of fast solvers. {\sc Li, Ming and Shi}~\cite[Lemma 4.1]{LiMingShi:2020} have constructed an enriching operator for the quadratic Specht triangle and have obtained optimal error estimate for approximating the biharmonc problems with rough solution. The construction and the proof therein equally applies to the Specht triangle. We also note that there are many different type enriching operators for Morley's triangle; See, e.g.,~\cite{Gallistl:2015, Veeser:2019}. In what follows, we shall construct such enriching operator for the NZT tetrahedron with the aid of the ninth polynomial $C^1$-conforming element introduced by {\sc Zhang}~\cite{Zhang:2009}. The enriching operator $E_h{:}X_h^{\,0}\rightarrow H^2_0(\Om)$ is defined as follows.
\begin{enumerate}
\item For any $a\in\mc{V}_h^{\,I}$, we fix a element $K_a$ from element star $\mc{T}_h(a)$, 
\[
\lr{\na^{\al}E_hv}(a){:}=\lr{\na^{\al}v|_{K_a}}(a),\, K_a\in\mc{T}_h(a),\,\abs{\al}\le 4.
\]

\item For any edge $e\in\mc{E}_h^{\,I}$, we equip the edge $e$ with unit direction vectors $\mc{S}_e=\{s_1,s_2\}$, where $s_i$ is orthogonal to $e$, and $\mc{S}_e\cup\{e\}$ provides a basis of $\mb{R}^3$. We fix an element from element star $K_e\in\mc{T}_h(e)$. For $a$ the middle point of an edge $e$, 
\begin{equation}\label{eq:edge1}
\pa_{s_i}E_hv(a){:}=\lr{\pa_{s_i}v|_{K_e}}(a),\quad s_i\in\mc{S}_e.
\end{equation}
For $b$ and $c$ the equally distributed interior points of the edge $e$, for $p=b,c$, 
\begin{equation}\label{eq:edge2}
\dfrac{\pa^2 E_hv}{\pa s_i\pa s_j} (p)=\dfrac{\pa^2 v|_{K_e}}{\pa s_i\pa s_j}(p),\quad s_i,s_j\in\mc{S}_e.
\end{equation}

\item For any $f\in\mc{F}_h^{\,I}$, and for any $w\mb{P}_0(f)$, we define
\[
\int_fE_hvw\,\md\sigma(x){:}=\int_fvw\,\md\sigma(x).
\]
and for any $w\in\mb{P}_2(f)$,
\begin{equation}\label{eq:moment}
\int_f\pa_n\lr{E_hv}w\,\md\sigma(x){:}=\int_f\{\!\{\pa_nv\}\!\}w\,\md\sigma(x).
\end{equation}

\item For any $K\in\mc{T}_h$, and for $w\in\mb{P}_1(K)$,
\[
\int_KE_hvw\dx{:}=\int_Kvw\dx.
\]
\item All the degrees of freedom of $E_h v$ vanish on the $\pa\Om$.
\end{enumerate}
We summarize the properties of the enriching operator in the following lemma.
\begin{lemma}
The enriching operator $E_h$ has the following properties:
\begin{enumerate}
\item Petrov-Galerkin orthogonality: For any $v\in V_h$,
\begin{equation}\label{eq:Galerkinorth}
a_h(v-E_hv,w)=0\qquad\text{for all\quad}w\in W_h,
\end{equation}
where $W_h$ is the tensorized $\mb{P}_2$ element space with vanishing trace.

\item $E_h$ is stable in the sense that
\begin{equation}\label{eq:stable}
\wnm{E_h v}\le \alpha\wnm{v}\qquad\text{for all\quad}v\in V_h.
\end{equation}

\item For any $v\in V_h$, we have
\begin{equation}\label{eq:l2enrich}
\nm{\na_h^k(v-E_h v)}{L^2}\le\beta h^{j-k}\nm{\na_h^jv}{L^2},\quad 0\le k\le j\le2.
\end{equation}
\end{enumerate}
\end{lemma}
The stability estimate~\eqref{eq:stable} is a direct consequence of~\eqref{eq:l2enrich}.
We only give the details for~\eqref{eq:Galerkinorth} and~\eqref{eq:l2enrich}.
\begin{proof}
For any $v\in V_h$ and $w\in W_h$, integration by parts, we obtain
\begin{align*}
&a_h(v-E_hv,w)=\sum_{K\in\mc{T}_h}\int_{\pa K}\pa_n(v-E_hv)M_{nn}(w)\md\sigma(x)\\
&\qquad=\sum_{F\in\mc{F}_h^{\,I}}\int_{F}\lr{[\![\pa_n(v-E_hv)]\!]\{\!\{M_{nn}(w)\}\!\}+\{\!\{\pa_n(v-E_hv)\}\!\}[\![M_{nn}(w)]\!]}\md\sigma(x).
\end{align*}
Using the fact that $M_{nn}(w)=n^{\mathrm{T}}\cdot\lr{\mb{D}\na\eps(w)}\cdot n\in\mb{P}_0(K)$,
\[
\jump{\pa_n(v-E_hv)}=\jump{\pa_nv}\quad\text{and}\quad
\{\!\{\pa_n(v-E_hv)\}\!\}=\{\!\{\pa_nv\}\!\}-\pa_nE_hv,
\]
and~\eqref{eq:moment}, we obtain~\eqref{eq:Galerkinorth}.

As to~\eqref{eq:l2enrich}, we only prove $k=0$, and the general case may be resorting to the inverse inequality because $v$ and $E_h v$ are piecewise polynomials.

For any element $K\in\mc{T}_h$, we let $\mc{N}, \mc{E}, \mc{F}$ and $\mc{V}$ be the set of the nodal variables, edge variables, face variable, and the set of the volume variables of $\mb{P}_9$ conforming element, respectively. For any $v\in V_h$, $v-E_h v\in\mb{P}_9$, and it follows from the scaling argument that
\begin{align*}
\nm{v-E_h v}{L^2(K)}^2&\le C\sum_{N\in\mc{N}(K)}h_K^{3+2\text{order}(N)}\Lr{N(v-E_h v)}^2\\
&\quad+C\sum_{E\in\mc{E}(K)}h_K^{3+2\text{order}(E)}\Lr{E(v-E_h v)}^2\\
&\quad+C\sum_{F\in\mc{F}(K)}h_K^{3+2\text{order}(F)}\Lr{F(v-E_h v)}^2\\
&\quad+C\sum_{V\in\mc{V}(K)}h_K^{3+2\text{order}(V)}\Lr{V(v-E_h v)}^2\\
&={:}I_1+\cdots+I_4,
\end{align*}
where $\mathrm{order}(N)$ is the order of the differentiation in the definition of $N$, and the same rule applies to $\mathrm{order}(E)$, $\mathrm{order}(F)$ and $\mathrm{order}(V)$. It is clear that $I_4=0$ because $V(v)=V(E_hv)$.

Note that we have $N(v-E_h v)=0$ when $\mathrm{order}(N)=0,1$. It remains to estimate the case $\mathrm{order}(N)=2,3,4$. By the standard inverse estimate, we obtain
\begin{align*}
I_1&\leq C\sum_{l=2}^4\sum_{p\in\mc{V}_K}h_{K'}^{3+2l}\sum_{K'\in\mc{T}_h(p)}\nm{\na^lv}{L^{\infty}(K')}^2\\
&\leq C\sum_{l=2}^4\sum_{p\in\mc{V}_K}h_{K'}^{3+2l}\sum_{K'\in\mc{T}_h(p)}h_{K'}^{1-2l}\nm{\na^2v}{L^2(K')}^2\\
&\leq C\sum_{K'\in\mc{T}_K}h_{K'}^4\nm{\na^2_hv}{L^2(K')}^2,
\end{align*}
where $\mc{V}_K$ is the set containing all the vertices of $K$, and $\mc{T}_K=\cup_{p\in\mc{V}_K}\mc{T}_h(p)$ is local element star of $K$.

Next, we denote $I_2^1$ for the term in $I_2$ with $\mathrm{order}(E)=1$, and $I_2^2$ for the term in $I_2$ with $\mathrm{order}(E)=2$. For $\mathrm{order}(E)=1$ and any edge $e\in\mc{E}_h$, by the ~\eqref{eq:edge1}, we have
\[
\abs{E(v-E_hv)}^2\le\abs{\lr{\na v|_K}(a)-\lr{\na v|_{K_e}}(a)}^2.
\]
We select a sequence of elements $\{K_1,\cdots,K_{J_e}\}\subset\mc{T}_h(e)$ such that $K_1=K,K_{J_e}=K_e$, and $f_j=K_j\cap K_{j+1}$ is a common face of $K_j$ and $K_{j+1}$.  We write the right-hand side of the above inequality as the telescopic sum:
\begin{align*}
\labs{(\na v|_K)(a)-(\na v|_{K_e})(a)}^2&=\sum_{j=1}^{J_e-1}\labs{\pa_nv|_{K_j}(a)-\pa_nv|_{K_{j+1}}(a)}^2\\
&\le C\sum_{f\in\mc{F}_h(e)}\abs{f}^{-1}\nm{\jump{\pa_nv}}{L^2(f)}^2.
\end{align*}
Note that  $\int_f\jump{\pa_n v}\dsx=0$, we use the Poincar\'{e} inequality and the trace inequality ~\eqref{eq:polytrace} to obtain
\begin{align*}
I_2^1&\le C\sum_{e\in\mc{E}_K}\sum_{f\in\mc{F}_h(e)}h_K^5\abs{f}^{-1}\nm{\jump{\pa_n v}}{L^2(f)}^2\\
&\le C\sum_{e\in\mc{E}_K}\sum_{f\in\mc{F}_h(e)}h_K^5\nm{\na\jump{\pa_n v}}{L^2(f)}^2\\
&\le C\sum_{K'\in\mc{T}_K}h_K^4\nm{\na^2v}{L^2(K')}^2,
\end{align*}
where $\mc{E}_K$ is the set containing all edges of $K$.

For $b$ and $c$ the equally distributed interior points of an edge $e$, and $p=b,c$, by~\eqref{eq:edge2} and the inverse estimate, we obtain
\[
I_2^2\le C\sum_{e\in\mc{E}_K}\sum_{K'\in\mc{T}_h(e)}h_{K'}^7\nm{\na^2v}{L^{\infty}(K')}^2
\le C\sum_{K'\in\mc{T}_K}h_{K'}^4\nm{\na^2v}{L^2(K')}^2.
\]

Last,  we have $F(v-E_h v)=0$ when $\mathrm{order}(F)=0$. By~\eqref{eq:moment}, and a standard scaling argument gives
\[
I_3\le C\sum_{f\in\mc{F}_K}h_F^5\labs{\negint_f\jump{\pa_n w}\md\sigma(x)}^2\le C\sum_{f\in\mc{F}_K}h_f^5\abs{f}^{-1}\nm{\jump{\pa_n w}}{L^2(f)}^2,
\]
where $\mc{F}_K$ is the set containing all face of $K$. Note that  $\int_f\jump{\pa_nv}\md\sigma(x)=0$ for any face $f\in\mc{F}_h$, using Poincar\'{e} inequality and the trace inequality again, we obtain
\[
I_3\le C\sum_{K'\in\mc{T}_K}h_{K'}^4\nm{\na^2v}{L^2(K')}^2.
\]

Summing up the estimates for $I_1,\cdots,I_4$, we obtain
\begin{align*}
\nm{v-E_h v}{L^2(K)}^2&\le C \sum_{K'\in\mc{T}_K}h_{K'}^4\nm{\na^2 v}{L^2(K')}^2\le C\sum_{K'\in\mc{T}_K}h_{K'}^{2j}\nm{\na^j v}{L^2(K')}^2.
\end{align*}
Summing up all the elements for $K\in\mc{T}_h$, we obtain~\eqref{eq:l2enrich}. 
\end{proof}
\begin{theorem}\label{thm:conv}
Let $u$ and $u_h$ be the solutions of Problem~\eqref{variation:eq} and Problem~\eqref{eq:disvara}, respectively. Then
\begin{equation}\label{eq:h2err}
\wnm{u-u_h}\le(1+\alpha)\Lr{2\inf_{v\in V_h}\wnm{u-v}+\inf_{w\in W_h}\wnm{u-w}}+\beta\,\emph{\osc}(f),
\end{equation}
where the oscillation of $f$ is defined as
\[
\emph{\osc}(f){:}=\Lr{\sum_{K\in\mc{T}_h}h_K^2\inf_{\ov{f}\in\mb{P}_1(K)}\nm{f-\ov{f}}{L^2(K)}^2}^{1/2}.
\]
\end{theorem}
\begin{proof}
For any $v\in V_h$, we denote $w=v-u_h$ and $E_h w=(E_h w_1,\cdots,E_h w_d)$. By the Galerkin orthogonality of the enriching operator~\eqref{eq:Galerkinorth}, we obtain, for any $z\in W_h$,
\begin{equation}\label{eq:errcore}
\begin{aligned}
\wnm{w}^2&=a_h(v,w)-a_h(u_h,w)=a_h(v,w-E_hw)+a_h(v,E_h w)-(f,w)\\
&=a_h(v-z,w-E_hw)+a_h(v-u,E_hw)+(f,E_h w-w)\\
&=a_h(v-z,w-E_hw)+a_h(v-u,E_hw)+(f-\ov{f},E_h w-w),
\end{aligned}
\end{equation}
where we have used~\eqref{eq:moment}$_3$ in the last step. The energy estimate~\eqref{eq:h2err} follows from~\eqref{eq:stable} and~\eqref{eq:l2enrich} with $k=0,j=1$ and the triangle inequality and the estimate
\[
\abs{a_h(v-z,w-E_hw)}\le(1+\al)\lr{\wnm{u-v}+\wnm{u-z}}\wnm{w}.
\]
\end{proof}

We are ready to derive the rate of convergence for the Specht triangle and the NZT tetrahedron.
\begin{theorem}
Let $u$ and $u_h$ be the solutions of problem~\eqref{variation:eq} and problem~\eqref{eq:disvara}, if the Hypothesis 2.3 is true, then
\begin{equation}\label{estimate:1}
\wnm{u-u_h}\le Ch^{1/2}\nm{f}{L^2}.
\end{equation}
If $u\in H^3(\Om)$, then
\begin{equation}\label{estimate:2}
\wnm{u-u_h}\le C(h^2+\iota^2 h)\nm{u}{H^3}.
\end{equation}
\end{theorem}
\begin{proof}
By~\eqref{interpolation:2}, and combining the regularity estimates~\eqref{eq:estfrac1} and ~\eqref{eq:regfrac2}, we have
\[
\inf_{v\in V_h}\wnm{u-v}\le\iota\nm{\na^2_h(u-I_h^{\,0}u)}{L^2}+\nm{\na(u-I_h^{\,0}u)}{L^2}\le Ch^{1/2}\nm{f}{L^2}.
\]

By the interpolation error estimate of {\sc Scott-Zhang} interpolant~\cite{ScottZhang:1990}, we obtain
\[
\inf_{v\in W_h}\wnm{u-v}\le\iota\nm{\na^2_h(u-\vpi_Cu)}{L^2}+\nm{\na(u-\vpi_Cu)}{L^2}\le Ch^{1/2}\nm{f}{L^2}.
\]
It is clear that
\[
\osc(f)\le Ch\nm{f}{L^2}.
\]
Combining all the above inequalities and~\eqref{eq:h2err}, we obtain~\eqref{estimate:1}.

The estimate~\eqref{estimate:2} may be proved in a standard manner, we omit the details.
\end{proof}
\section{Numerical Experiments}
In this part, we test the accuracy of the Specht triangle and the NZT tetrahedorn and the numerical pollution effect for a solution with strong boundary layer. In all the examples, we let $\Om=(0,1)^d$ and set $\lam=10,\mu=1$. For $d=2$, the initial unstructured mesh consists of $220$ triangles and $127$ vertices, and the maximum mesh size is $h=1/8$; See Figure~\ref{mesh}$_ a$. For $d=3$, we construct initial mesh by splitting origin cube into $512$ small cubes, and each small cube is divided into 6 tetrahedrons; See Figure~\ref{mesh}$_b$.
\begin{figure}[htbp]
\centering
\subfigure[]{\includegraphics[width=6cm, height=4.5cm]{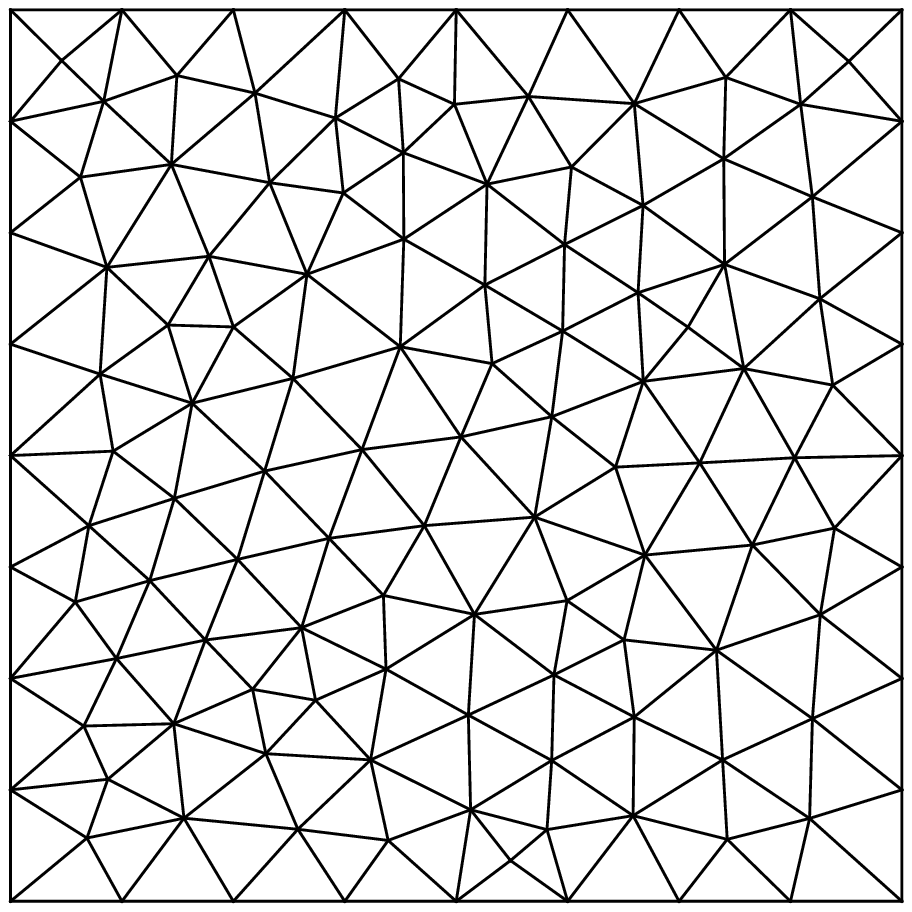}}
\subfigure[]{\includegraphics[width=6cm, height=4.5cm]{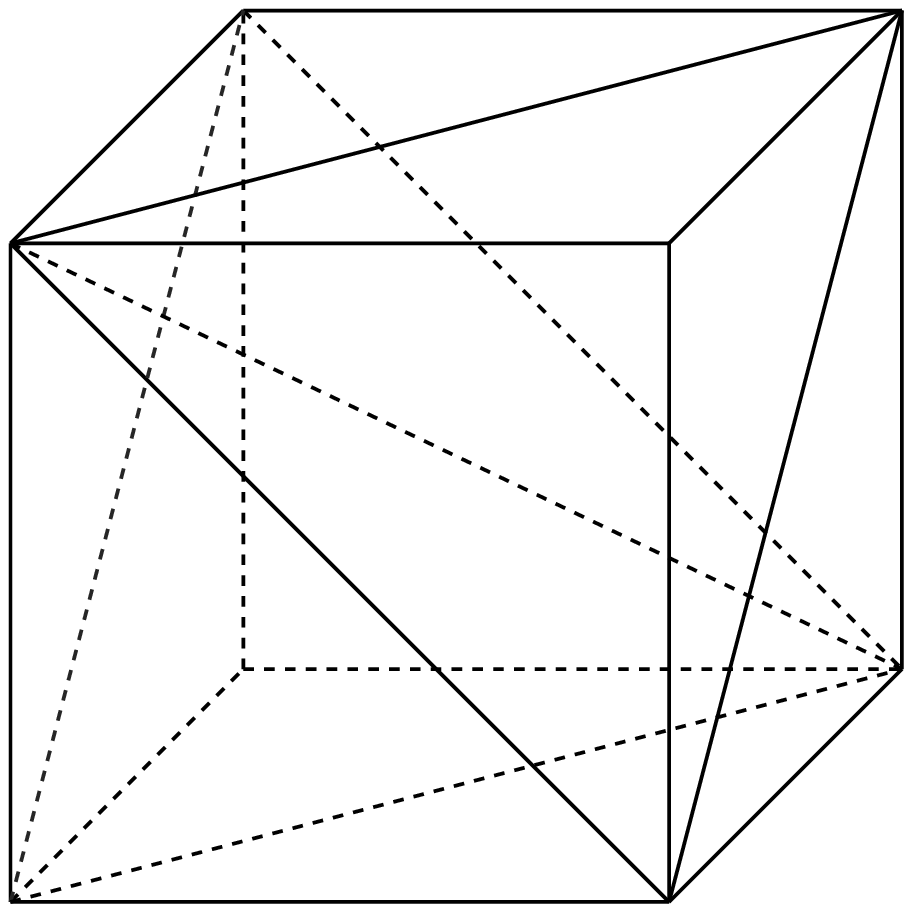}}
\caption{Plots of meshes: (a) $d=2$; (b) $d=3$.}\label{mesh}
\end{figure}

Throughout the simulation, we employ the analytic basis functions for the Specht triangle~\cite{Specht:1988, ZienkiewiczTaylor:2009} and the NZT tetrahedron~\cite{WangShiXu:2007}. The computation for the NZT tetrahedron is performed in a parallel hierarchical grid platform (PHG)~\cite{zlb:2009}.\footnote{http://lsec.cc.ac.cn/phg}
\subsection{Example for smooth solution}
This example is to test the accuracy of the elements. We let $u=(u_1,u_2,u_3)$ with
\begin{align*}
u_1&=\prod_{i=1}^d\lr{\exp(\cos2\pi x_i)-\exp(1)},u_2=\prod_{i=1}^d(\cos2\pi x_i-1),\\
u_3&=\prod_{i=1}^dx_i^2(x_i-1)^2.
\end{align*}
The source term $f$ is computed by~\eqref{eq:sgbvp}$_1$. For $d=2$, we drop the third component $u_3$.
This solution is smooth, and we measure the rates of convergence in the relative energy norm $\wnm{u-u_h}/\wnm{u}$ for different $\iota$.  We report the rates of convergence for the Specht triangle and the NZT tetrahedorn in Table~\ref{tab:1} and Table~\ref{tab:2}, respectively. We observe that the rates of convergence appear to be linear when $\iota$ is large, while it turns out to be quadratic when $\iota$ is close to zero, which is consistent with the theoretical predications~\eqref{estimate:2}.
\begin{table}[htbp]
\caption{Rate of convergence of the Specht triangle.}\label{tab:1}
\begin{tabular}{lllllll}
\hline\noalign{\smallskip}
$\iota\backslash h$ &$1/8$ &$1/16$ &$1/32$ &$1/64$ &$1/128$ &$1/256$\\
\hline\noalign{\smallskip}
1e+0 &1.99e-01&9.87e-02&4.80e-02&2.36e-02&1.17e-02&5.85e-03\\
\noalign{\smallskip}
\text{rate}&&1.01&1.04&1.02&1.01&1.00\\
\noalign{\smallskip}
1e-2 &3.16e-02&1.21e-02&5.30e-03&2.53e-03&1.25e-03&6.21e-04\\
\noalign{\smallskip}
\text{rate}&&1.39&1.19&1.07&1.02&1.01\\
\noalign{\smallskip}
1e-4 &2.20e-02&5.57e-03&1.39e-03&3.48e-04&8.75e-05&2.26e-05\\
\noalign{\smallskip}
\text{rate}&&1.98&2.00&2.00&1.99&1.95\\
\noalign{\smallskip}
1e-6 &2.20e-02&5.57e-03&1.39e-03&3.47e-04&8.65e-05&2.16e-05\\
\noalign{\smallskip}
\text{rate}&&1.98&2.00&2.00&2.00&2.00\\
\noalign{\smallskip}
\hline
\end{tabular}
\end{table}
\begin{table}[htbp]
\caption{Rates of convergence of the NZT tetrahedorn.}\label{tab:2}
\begin{tabular}{llllll}
\hline\noalign{\smallskip}
$\iota\backslash h$ &$1/8$ &$1/16$ &$1/32$ &$1/64$ &$1/128$\\
\hline\noalign{\smallskip}
1e+0 & 1.02e-01& 7.54e-01& 5.05e-01& 2.84e-01& 1.47e-01\\
\noalign{\smallskip}
\text{rate}&&0.43 &0.58 &0.83 &0.95\\
\noalign{\smallskip}
1e-2 & 5.12e-01 &2.28e-01& 8.99e-02& 3.57e-02& 1.55e-02\\
\noalign{\smallskip}
\text{rate}&&1.17 &1.34 &1.33 &1.20\\
\noalign{\smallskip}
1e-4 &3.03e-01 &7.14e-02 &1.79e-02 &4.66e-03 &1.27e-03\\
\noalign{\smallskip}
\text{rate}&&2.08 &1.99 &1.94 &1.89\\
\noalign{\smallskip}
1e-6 &3.01e-01 &6.98e-02 &1.71e-02 &4.26e-03 &1.07e-03\\
\noalign{\smallskip}
\text{rate}&&2.11 &2.03 &2.00 &1.99\\
\hline
\end{tabular}
\end{table}
\subsection{Example with boundary layer}
In this example, we test the performance of both elements to resolve a solution with strong boundary layer, such boundary layer is one of the difficulty for the strain gradient elasticity model, and we refer to~\cite{Engel:2002} for a one-dimensional example with analytical expression, the following construction is based on that example. We construct a displacement field $u=(u_1,u_2,u_3)$ with a layer as
\begin{align*}
u_1&=\prod_{i=1}^d(\exp(\sin\pi x_i)-1-\varphi(x_i)),\quad u_2=\prod_{i=1}^d(\sin\pi x_i-\varphi(x_i)),\\
u_3&=\prod_{i=1}^d\lr{-\pi x_i(x_i-1)-\varphi(x_i)}
\end{align*}
with
\[
\varphi(x)=\pi\iota\dfrac{\cosh[1/2\iota]-\cosh[(2x-1)/2\iota]}{\sinh[1/2\iota]}.
\]
A direct calculation gives
\[
\lim_{\iota\rightarrow0}u=u_0=\Lr{\prod_{i=1}^d\exp(\sin\pi x_i)-1,\prod_{i=1}^d\sin\pi x_i,\prod_{i=1}^d\pi x_i(1-x_i)},
\]
with $u_0|_{\pa\Om}=0$ and $\pa_nu_0|_{\pa\Om}\neq0$. It is clear that $\pa_n u$ has boundary layers.
The source term $f$ is also computed from~\eqref{eq:sgbvp}$_1$. We report the rates of convergence for the elements in the relative energy norm $\wnm{u-u_h}/\wnm{u}$ with $\iota=10^{-6}$ in Table~\ref{tab:3a}. The half order rates of convergence are observed for both elements, which are consistent with the theoretical predictions~\eqref{estimate:2}.
\begin{table}[htbp]
\caption{Rates of convergence for $\iota=10^{-6}$.}\label{tab:3a}
\begin{tabular}{lllllll}
\hline\noalign{\smallskip}
$h$ &$1/8$ &$1/16$ &$1/32$ &$1/64$ &$1/128$ &$1/256$ \\
\hline\noalign{\smallskip}

NZT &2.74e-01 &1.73e-02 &1.18e-01 &8.24e-02 &5.80e-02 &4.10e-2\\
\noalign{\smallskip}
\text{rate}&&0.66 &0.56 &0.52 &0.51 &0.50\\
\noalign{\smallskip}

Specht&1.57e-01&1.10e-01&7.70e-02&5.42e-02&3.82e-02&2.70e-02\\
\noalign{\smallskip}
\text{rate}&&0.51&0.51&0.51&0.50&0.50\\
\noalign{\smallskip}
\hline
\end{tabular}
\end{table}
\section{Conclusion}
We prove a new H$^2-$Korn's inequality and a new broken H$^2-$Korn's inequality. The former is crucial for the well-posedness of a strain gradient elasticity model, while the latter motivates us to construct robust nonconforming elements for this model, and the elements are simpler than the known elements in the literature; See, e.g.,~\cite{LiMingShi:2017}. With the aid of the new regularized interpolant and the enriching operator, we proved that the tensor product of the Specht triangle and the NZT tetrahedron converges uniformly with respect to the small materials parameter under the minimal smoothness assumption on the solution. Moreover, the technicalities may also be used to derive shaper error bounds for the elements in~\cite{Tai:2001, Tai:2006, GuzmanLeykekhmanNeilan:2012}. Guided by the broken H$^2-$Korn's inequality, we can design robust elements for the nonlinear strain gradient elastic models, thin beam and thin plate with strain gradient effect in~\cite{Fisher:2010, Engel:2002} by combining the tricks in~\cite{BraessMing:2005, MingShi:2006} and the machinery developed in this article, which will be left for further pursuit.
\bibliographystyle{amsplain}
\bibliography{sg}
\end{document}